\documentclass[reqno,a4paper,12pt]{amsart}
\usepackage{amsmath,amsthm,amsfonts,amssymb, mathrsfs}
\usepackage{verbatim, graphicx, ifthen, enumitem}
\usepackage[T1]{fontenc}
\usepackage [applemac] {inputenc}
\usepackage{color}

\allowdisplaybreaks

\let\oldmarginpar\marginpar
\renewcommand\marginpar[1]{\-\oldmarginpar[\raggedleft\footnotesize #1]%
{\raggedright\footnotesize #1}}
\newtheorem{theorem}{Theorem}[section]
\newtheorem*{theorem*}{Theorem}
\newtheorem{lemma}[theorem]{Lemma}

\theoremstyle{definition}
\newtheorem{definition}[theorem]{Definition}
\newtheorem{example}[theorem]{Example}

\newtheorem{remark}[theorem]{Remark}

\newcommand{\Z}{\mathbb{Z}}

\newcommand{\N}{\mathbb{d}}
\newcommand{\R}{\mathbb{R}}
\newcommand{\T}{\mathbb{T}}

\def\L{\Lambda}
\def\l{\lambda}
\def\T{\mathbb{T}}
\def\N{\mathbb{N}}
\def\Z{\mathbb{Z}}

\def\R{\mathbb{R}}
\def\R{\mathbb{R}}

\def\1{\mathbf{1}}

\title{On uniformly minimal and 'uniformly complete' exponential systems}
\author{Shahaf Nitzan}
\address{Georgia Institute of Technology, Atlanta, USA}
\email{shahaf.nitzan@math.gatech.edu}
\thanks{The author is supported by NSF CAREER grant DMS 1847796}

\begin{document}

\begin{abstract}
A.~Olevskii and A.~Ulanovskii obtain in \cite{OU2010} a scale of density results, which correspond to how well an exponential system approximates a uniformly minimal system over a compact set. We extend their result in several directions. First, we show that it holds for any set of positive finite measure. Next, we consider a relaxed version of frames, which we term 'uniformly complete systems', and obtain an analogues scale of density results for such systems.
\end{abstract}

\maketitle

\section{introduction}

\subsection{}

Let $S\subseteq\R$ be of positive finite measure, and let $\L\subseteq\R$ be uniformly discrete, that is,
\begin{equation}\label{unif discrete}
|\l-\gamma|\geq \delta\qquad \forall\l\neq\gamma\in\L,
\end{equation}
where $\delta$ is a positive constant. Denote \[e_{\l}(t)=e^{2\pi i \l t},\] and consider the exponential system
\[
E(\L):=\{e_{\l}\}_{\l\in\L}
\]
in the space $L^2(S)$.

We say that
$E(\L) $
is a \textit{Riesz sequence} in the space if there exist positive constants $A$ and $B$ so that
\[
A\sum_{\l\in\L}|a_{\l}|^2\leq \|\sum_{\l\in\L} a_{\l}e_{\l}\|^2_{L^2(S)}\leq B\sum_{\l\in\L}|a_{\l}|^2\qquad \forall\{a_{\l}\}\in\ell^2(\L).
\]

It is well known that if $E(\L)$ is a Riesz sequence in $L^2(S)$ then it is \textit{uniformly minimal} there, or equivalently that there exists a sequence $\{g_{\l}\}_{\l\in\L}\subseteq L^2(S)$, with $\sup\|g_{\l}\|_{L^2(S)}< \infty$, which satisfies:
\begin{equation}\label{delta condition}
\langle  g_{\l}, e_{\gamma}\rangle_{L^2(S)} =\delta_{\l}(\gamma)\qquad \forall\l,\gamma\in\L,
\end{equation}
where $\delta_{\l}(\gamma)=1$ when $\l=\gamma$ and $\delta_{\l}(\gamma)=0$ otherwise.

In general, a uniformly minimal exponential system need not be a Riesz sequence: An example is well known, and discussed in Subsection \ref{subsection 3.2}. The notion of 'uniform minimality' may therefore be considered as an intermediate notion, between that of 'minimal systems' and that of 'Riesz sequences'.

\subsection{}\label{section 1.2} We say that the exponential system $E(\L)$ is a \textit{frame} in $L^2(S)$ if there exist positive constants $A$ and $B$ so that
\[
A\|f\|_{L^2(S)}^2 \leq \sum_{\l\in\L}|\langle f, e_{\l}\rangle_{L^2(S)}|^2\leq B\|f\|_{L^2(S)}^2\qquad \forall f\in L^2(S).
\]

Frames and Riesz sequences are traditionally considered dual notions, and their definitions correspond to the actions of dual operators. This duality is realized in several well known results, see e.g. \cite{MM2010}, \cite{RS1997}. These systems represent the dual problems of sampling and interpolation, as is discussed in more detail in e.g. \cite{OUbook}, \cite{Sbook} (we will not use this terminology in the current note).

It follows from the definition that if $E(\L)$ is a frame in $L^2(S)$ then it is complete in the space. In fact, it is well known that if $S$ is bounded, then $E(\L)$ is a frame in the space if and only if every $f\in L^2(S)$ admits a decomposition $f=\sum a_{\l}(f)e_{\l}$ with
\begin{equation}\label{frame decomposition}
\sum_{\l\in\L} |a_{\l}(f)|^2\leq C\|f\|_{L^2(S)}^2,
\end{equation}
where $C$ is a positive constant. (Note that Lemma {\ref{bessel lemma}} was applied for this equivalence. It remains true also for unbounded sets under the added condition that the conclusion of Lemma {\ref{bessel lemma}} holds.)

We are not aware of a notion from the literature which serves as intermediate between 'complete systems' and 'frames', and which may provide a good analog to the notion of 'uniform minimality'. One goal of this note is to suggest such a notion for exponential systems (see Subsection \ref{subsection 1.3}).

\subsection{} Intuitively, one may assume that if $E(\L)$ is a Riesz sequence then $\L$ cannot be too dense, while if it is a frame then $\L$ cannot be too sparse. This intuition was conjectured by A.~Beurling (who studied in detail the case where $S$ is an interval) and confirmed by H.~Landau, \cite{Lan1967}.

In his classical theorem, Landau proves that if $S$ is a bounded set of positive measure and $E(\L)$ is a Riesz sequence in $L^2(S)$ then
\begin{equation}\label{landau riesz}
D^+(\L):= \lim_{r\rightarrow\infty} \frac{\max_{x\in\R}\sharp\{\L\cap [x,x+r]\}}{r}\leq |S|,
\end{equation}
while if $E(\L)$ is a frame in the space then
\begin{equation}\label{landau frame}
D^-(\L):= \lim_{r\rightarrow\infty} \frac{\min_{x\in\R}\sharp\{\L\cap [x,x+r]\}}{r}\geq |S|.
\end{equation}
Since its publication in 1967, Landau's theorem inspired a variety of research in the theory of sampling and interpolation, both in this and in other settings.

 In the context of this note, the works \cite{NO2012} and \cite{OU2010} are of interest. In \cite{NO2012}, A.~Olevskii and the author show that the requirement that $S$ is bounded may be removed, and that the weaker condition of $E(\L)$ being uniformly minimal is enough to imply the density restriction (\ref{landau riesz}). (This is no longer true if $S$ is merely minimal, see \cite{OU2010}).

In \cite{OU2010}, A.~Olevskii and A.~Ulanovskii obtain a scaled version of the latter result, which they show to hold when $S$ is compact. Our first goal is to extend their result to general sets $S$ of positive finite measure. To give an explicit formulation we need the following.

\begin{definition}\label{app unif min}

Fix $0< d< 1$. Let $S\subseteq\R$ be of positive finite measure and let $\L\subseteq\R$ be a uniformly discrete sequence. We say that $E(\L)$ is $d$-\textit{approximately uniformly minimal} in $L^2(S)$ if there exist $\{g_{\lambda}\}_{\lambda\in\Lambda}\subseteq L^2(S)$, with $\sup\|g_{\lambda}\|_{L^2(S)}<\infty$, so that
\begin{equation}\label{unif-min-up-to-d}
\|\widehat{g}_{\l}-\delta_{\lambda}\|_{\ell^2(\L)}\leq d\qquad \forall\l\in\L,
\end{equation}
where $\widehat{g}_{\l}(\gamma)= \langle  g_{\l}, e_{\gamma}\rangle_{L^2(S)}$ is the Fourier transform of $g_{\l}$, restricted to $\L$.
\end{definition}

Note that if the conditions of Definition \ref{app unif min} hold with $d=0$, then the system is uniformly minimal. With this definition the estimate (\ref{landau riesz}) may be quantified. We prove the following in Section \ref{section 2}.

\begin{theorem}\label{d-almos-unif-min}
Fix $0< d< 1$. Let $S\subseteq\R$ be of positive finite measure and let $\L\subseteq\R$ be a uniformly discrete sequence. If $E(\L)$ is $d$-approximate uniformly minimal in $L^2(S)$ then
\[
D^+(\L)\leq \frac{|S|}{1-d^2}.
\]
\end{theorem}

As mentioned above, the case where $S$ is compact was treated by Olevskii and Ulanovskii in \cite{OU2010}. Moreover, they show in \cite{OU2010} that this result is sharp in the sense that for every $0< d< 1$ there exists a sequence $\L$, so that $E(\L)$ is $d$-approximate uniformly minimal in $L^2[0,1]$ and $D^+(\L)= 1/({1-d^2})$.

\subsection{}\label{subsection 1.3} We may summarize the preceding discussion as follows: Landau obtained the estimate (\ref{landau riesz}) for Riesz sequences, it turns out that this estimate holds for uniformly minimal systems, but not for systems which are merely minimal. Moreover, the relation between the density restriction and the approximation of uniform minimality may be scaled.

Next, we ask whether the dual density restriction (\ref{landau frame}) admits analogous results. It is well known that condition (\ref{landau frame}) is not necessarily implied by the weaker requirement that $E(\L)$ is complete. Results go back to H.~Landau \cite{Lan64}, see also e.g. \cite{NO2007}.

As mentioned in Subsection \ref{section 1.2}, we are not familiar with a notion in the literature which may serve as an intermediate notion between 'complete systems' and 'frames', in the same way that the notion of 'uniform minimality' serves between 'minimal systems' and 'Riesz sequences'. For exponential systems, we propose to consider the following notion.

\begin{definition}\label{def unif comp}

Let $S\subseteq\R$ be of positive finite measure and let $\L\subseteq\R$ be a uniformly discrete sequence. We say that $E(\L)$ is \textit{uniformly complete} in $L^2(S)$ if for every $w\in\R$, and every $\epsilon>0$, there exists a linear combination
\[
p_w=\sum_{\l\in\L} a_{\l}(w)e_{\l},
\]
 so that
$
\|e_{w}-p_w\|_{L^2(S)}< \epsilon
$
and
\[
\sum_{\l\in\L} |a_{\l}(w)|^2\leq C,
\]
where $C$ is a positive constant which does not depend on $w$ or on $\epsilon$.
\end{definition}

Observe that this definition relaxes the frame condition in the sense that the decomposition and the estimate in (\ref{frame decomposition}) are required to hold only for exponential functions.

Clearly, if an exponential system is uniformly complete in $L^2(S)$ then it is complete there. On the other hand, in Section \ref{section 3} we show that a uniformly complete system need not be a frame, and so may indeed be considered as an intermediate notion. Further, we motivate in that section the definition of uniformly complete systems as analogs of uniformly minimal systems via a well known duality argument.

It turns out that uniform completeness plays a role in the study of the density restriction  (\ref{landau frame}), which is similar to the role that uniform minimality plays in the study of (\ref{landau riesz}). Explicitly, we have the following.

\begin{theorem}\label{thm unif complete}
Let $S\subseteq\R$ be of positive finite measure and let $\L\subseteq\R$ be a uniformly discrete sequence. If $E(\L)$ is uniformly complete in $L^2(S)$ then
\[
D^-(\L)\geq |S|.
\]
\end{theorem}

Next we consider the following quantified version of uniform completeness.

\begin{definition}

Fix $0< d< 1$. Let $S\subseteq\R$ be of positive finite measure and let $\L\subseteq\R$ be a uniformly discrete sequence. We say that $E(\L)$ is \textit{$d$-approximately uniformly complete} in $L^2(S)$ if for every $w\in\R$ there exists a linear combination
\[
p_w=\sum_{\l\in\L} a_{\l}(w)e_{\l},
\]
 so that
$
\|e_{w}-p_w\|_{L^2(S)}\leq d\sqrt{S}
$
and
\[
\sum_{\l\in\L}|a_{\l}(w)|^2\leq C,
\]
where $C$ is a positive constant which does not depend on $w$.
\end{definition}

With this definition we show that Theorem \ref{thm unif complete} admits a scaled version, analogous to the scaled version in Theorem \ref{d-almos-unif-min}.

\begin{theorem}\label{d-almos-unif-comp}
Fix $0< d< 1$. Let $S\subseteq\R$ be of positive finite measure and let $\L\subseteq\R$ be a uniformly discrete sequence. If $E(\L)$ is $d$-approximately uniformly complete in $L^2(S)$ then
\[
D^-(\L)\geq {|S|}({1-d^2}).
\]
This result is sharp (up to the end point) in the sense that for every $0< d< 1$ and for every $\epsilon>0$ there exists a set $S$ so that $E(\Z)$ is $d$-approximate uniformly complete in $L^2(S)$ and
$1-\epsilon\leq {|S|}({1-d^2})$.
\end{theorem}
    \vspace{5pt}
    
\begin{remark}
The notion of 'uniform completeness' defined here does not have an analog in an abstract Hilbert space setting. It may, however, be studied in the more general setting of RKHS, or in the setting of Gabor analysis. We discuss this in more detail in our upcoming work, joint with R.~Pai, \cite{NP}.
\end{remark}

\begin{remark}
In \cite{NO2007}, A.~Olevskii and the author study a different notion which is intermediate between 'complete systems' and 'frames'. It does not, however, serve as an analog to 'uniform minimality', and the density results related to it are significantly different from the results discussed above.
\end{remark}

The paper is organized as follows: In Section \ref{section 2} we prove Theorem  \ref{d-almos-unif-min}. In Section \ref{section 3} we motivate the definition of uniformly complete systems and prove that for exponential systems, the notion of 'uniform completeness' is indeed weaker then that of 'frames'. In Section \ref{section 4} we prove Theorem \ref{d-almos-unif-comp} and obtain Theorem \ref{thm unif complete} as a corollary.

\subsection{} The following notations will be fixed throughout the paper. We use the convention
\[
\widehat{f}(w)=\int_{\R}f(t)e^{-2\pi i wt}dt,
\]
for the Fourier transform of a function in $L^1(\R)$, with the standard extension to $L^2(\R)$.

For $\L\subseteq\R$ and $\delta>0$ we say that $\L$ is \textit{uniformly discrete with separation constant} $\delta$ if (\ref{unif discrete}) holds for $\L$ and $\delta$. For $\lambda\in\Lambda$ we denote by $\delta_{\lambda}$ the function defined on $\L$ which gives the value $1$ on $\l$ and zero otherwise.

We denote the cardinality of finite set $A$ by $\sharp A$, and the Lebesgue measure of a measurable set $S\subseteq \R$ by $|S|$. The indicator of such a set, $1\!\!1_S$, is equal to $1$ on $S$ and $0$ otherwise. We denote by $C$ different constants which may change from line to line.

\section{Approximate uniform minimality}\label{section 2}
In this section we prove  Theorem  \ref{d-almos-unif-min}.

\subsection{Auxiliary lemmas}

We start by reviewing two lemmas. The first may be viewed as a finite dimensional version of Theorem  \ref{d-almos-unif-min}. It is due to A.~Olevskii and A.~Ulanovskii, who applied it to study the case of compact sets $S$ in \cite{OU2010}. While it plays a different role in our proof of Theorem  \ref{d-almos-unif-min}, it still remains one of the main tools we apply. See \cite{OU2010} for the proof (or \cite{OUbook}, pp 78, Lemma 8.4).

Bellow we denote by $\textrm{id}_{N}$ the identity $N\times N$ matrix, and by $\|B\|_{HS}$ the Hilbert--Schmidt norm of an $N\times N$ complex valued matrix $B=(b_{jk})$, that is,
\[
\|B\|^2_{HS}=\sum_{j=1}^N\sum_{k=1}^N|b_{jk}|^2.
\]

\begin{lemma} \label{lin-algebra-lemma}
Let $B$ be an $N\times N$ complex valued matrix, and assume that for some $0<d<1$ we have
\[
\|B-\textrm{id}_{N}\|^2_{HS}\leq d^2N.
\]
Then for every $1<\gamma<1/d$ there exists a subspace $L\subseteq \ell_2^{N}:=\mathbb{C}^N$ of dimension
\[
(1-\gamma^2d^2)N-1\leq \textrm{dim}L,
\]
such that
\begin{equation}\label{matrix-bound}
\big(1-\frac{1}{\gamma}\big)\|a\|_{\ell_2^N}\leq \|Ba\|_{\ell_2^N} \qquad \forall a\in L.
\end{equation}
\end{lemma}

Next, we need the following lemma which was observed in \cite{ACNS} in the context of Gabor analysis. Below we adapt this observation to the setting of exponential systems, and add a short proof for completeness.

\begin{lemma}\label{children}
Let $S\subseteq\R$ be a set of positive finite measure and let $W\subseteq L^2(S)$ be a closed subspace. Denote by $P_W$ be the orthogonal projection from $L^2(S)$ onto $W$. Then,
\[
\int_{\R}\|P_We_x\|^2_{L^2(S)}dx=\textrm{dim}W.
\]
\end{lemma}
\begin{proof}
Denote dim$W=N$ and let $h_1,..,h_N$ be an orthonormal basis of $W$. We have
\[
\begin{aligned}
&\int_{\R}\|P_We_x\|_{L^2(S)}^2dx=\int_{\R}\sum_{n=1}^N|\langle h_n,e_x\rangle_{L^2(S)}|^2dx=\\
&=\sum_{n=1}^N\int_{\R}|\widehat{h}_n(x)|^2dx=\sum_{n=1}^N\|h_n\|_{L^2(S)}^2=N.
\end{aligned}
\]
\end{proof}

\subsection{A proof for Theorem \ref{d-almos-unif-min}}
\begin{proof}

Let $0<d<1$, $S\subseteq\R$, $\L\subseteq\R$ and $\{g_{\l}\}_{\l\in\L}\subseteq L^2(S)$ be as in Definition \ref{app unif min} and  Theorem \ref{d-almos-unif-min}. Denote
\begin{equation}\label{bound for g_l}
M:=\sup_{\l\in\L}\|g_{\l}\|_{L^2(S)}.
\end{equation}
Fix $0<\epsilon <1$ and $1<\gamma<\frac{1}{d}$. Denote $h:=\widehat{1\!\!1}_S$, and let $b>0$ be such that
\begin{equation}\label{b-size}
\int_{|x|>b}|h(x)|^2<\epsilon^2.
\end{equation}

\underline{\textit{Step} I.} Given an interval $I\subseteq\R$ put
\[
\Lambda_I=\Lambda\cap I\qquad \textrm{and}\qquad  N_I=\sharp\Lambda_I.
\]
Note that, since $\Lambda$ is uniformly discrete, $\Lambda_I$ is finite and we can write $\Lambda_I=\{\lambda_1<\lambda_2<...<\lambda_{N_I}\}$.

Next, consider the $N_I\times N_I$ complex valued matrix $B_I$ defined by
\[
(B_I)_{jk}=\langle {g_{\lambda_j}}, e_{\lambda_k}\rangle_{L^2(S)}\qquad j,k=1,..,N_I.
\]
Observe that for $a=(a_k)\in \ell_2^{N_I}$ we have
\[
\begin{aligned}
\|B_Ia\|_{\ell_2^{N_I}}^2&=\sum_{j=1}^{N_I}|\sum_{k=1}^{N_I} a_k\langle {g_{\lambda_j}}, e_{\lambda_k}\rangle_{L^2(S)}|^2\\
&=\sum_{j=1}^{N_I}|\langle {g_{\lambda_j}}, \sum_{k=1}^{N_I} \overline{a}_ke_{\lambda_k}\rangle_{L^2(S)}|^2,
\end{aligned}
\]
which implies that if $w=\sum \overline{a}_ke_{\l_k}$ then
\begin{equation}\label{mult by matrix}
\|B_Ia\|_{\ell_2^{N_I}}^2=\sum_{j=1}^{N_I}|\langle g_{\lambda_j}, w\rangle_{L^2(S)}|^2.
\end{equation}

Next, we wish to apply  Lemma \ref{lin-algebra-lemma} to the matrix $B_I$. For this we note that due to (\ref{unif-min-up-to-d}) we have
\[
\|B_I-\textrm{id}_{N_I}\|^2_{HS}\leq d^2N_I,
\]
and so Lemma \ref{lin-algebra-lemma} may indeed be applied. It follows that there exists a subspace $L\subseteq\ell_2^{N_I}$, of dimension
\begin{equation}\label{dim L}
(1-\gamma^2d^2)N_I-1\leq \textrm{dim}L,
\end{equation}
so that (\ref{matrix-bound}) holds for $B=B_I$ and every $a\in L$.

\underline{\textit{Step} II.} Denote,
\[
W=\{\sum_{k=1}^{N_I}\overline{a}_ke_{\lambda_k}: (a_k)\in L, \l_k\in\L_I\}.
\]
Since any finite family of exponential functions is linearly independent, we have $\textrm{dim}W=\textrm{dim}L$ and so it follows from (\ref{dim L}) that
\begin{equation}\label{dim W}
(1-\gamma^2d^2)N_I-1\leq \textrm{dim}W.
\end{equation}
Observe that for $w=\sum_{k=1}^{N_I}\overline{a}_ke_{\lambda_k}\in W$  the relations (\ref{matrix-bound}) and (\ref{mult by matrix}) imply that
\begin{equation}\label{RB-type-estimate}
(1-\frac{1}{\gamma})^2\sum_{k=1}^{N_I}|a_k|^2\leq \sum_{j=1}^{N_I}|\langle g_{\lambda_j},w\rangle_{L^2(S)}|^2.
\end{equation}

\underline{\textit{Step} III.} Let $P_W$ be the orthogonal projection from $L^2(S)$ onto $W$. For $x\in\R$ we have
\[
\|P_We_x\|_{L^2(S)}^2\leq \|e_x\|_{L^2(S)}^2=|S|.
\]
Denote by $I(b)$ the interval with the same center as $I$ and measure equal to $|I|+2b$ (where $b$ was defined in (\ref{b-size})). It follows that,
\[
\int_{I(b)}\|P_We_x\|_{L^2(S)}^2dx\leq |S|(|I|+2b).
\]
By Lemma \ref{children} we therefore have
\[
\textrm{dim}W\leq  |S|(|I|+2b)+\int_{\R\setminus I(b)}\|P_We_x\|_{L^2(S)}^2dx,
\]
which by the estimate (\ref{dim W}) implies that
\begin{equation}\label{use-of-observ}
(1-\gamma^2d^2)N_I-1\leq |S|(|I|+2b)+\int_{\R\setminus I(b)}\|P_We_x\|_{L^2(S)}^2dx.
\end{equation}
Denote the integral on the right side of (\ref{use-of-observ}) by $\Phi$. We next wish to estimate $\Phi$.

\underline{\textit{Step} IV.}
Let $x\in\R$. Since $P_We_x\in W$, there exist $(a_k(x))\in L$ so that $P_We_x=\sum_k\overline{a}_k(x)e_{\lambda_k}$. Recall the function $h$ which was defined in (\ref{b-size}). The Cauchy-Shwartz inequality implies that,
\[
\begin{aligned}
\|P_We_x\|_{L^2(S)}^2&=\langle P_We_x,e_x\rangle_{L^2(S)}=\langle \sum_{k=1}^{N_I} \overline{a}_k(x)e_{\lambda_k}, e_x\rangle_{L^2(S)}=\\
&=\sum_{k=1}^{N_I}\overline{a}_k(x)h(x-\lambda_k)\\
&\leq \big(\sum_{k=1}^{N_I}|a_k(x)|^2\big)^{\frac{1}{2}}\big(\sum_{k=1}^{N_I}|h(x-\lambda_k)|^2\big)^{\frac{1}{2}}.\\
\end{aligned}
\]
Applying (\ref{RB-type-estimate}) to the last expression we find that
\[
\|P_We_x\|_{L^2(S)}^2\leq C(\gamma)\big(\sum_{j=1}^{N_I}|\langle g_{\lambda_j}, P_W e_x\rangle_{L^2(S)}|^2\big)^{\frac{1}{2}}\big(\sum_{k=1}^{N_I} |h(x-\lambda_k)|^2\big)^{\frac{1}{2}},\\
\]
where $C(\gamma)$ is a constant depending only on $\gamma$. Integrating over $\R\setminus I(b)$ and applying the Cauchy-Shwartz inequality again, we obtain the estimate
\begin{equation}\label{estimate for I}
\Phi\leq C(\gamma)\Big(\sum_{j=1}^{N_I}\int_{\R\setminus I(b)}|\langle g_{\lambda_j}, P_W e_x\rangle_{L^2(S)}|^2dx\Big)^{\frac{1}{2}}\Big(\sum_{k=1}^{N_I}\int_{\R\setminus I(b)}|h(x-\lambda_k)|^2dx\Big)^{\frac{1}{2}}.
\end{equation}
For $j,k=1,...,N_I$ denote
\[
\alpha_j=\int_{\R\setminus I(b)}|\langle g_{\lambda_j}, P_W e_x\rangle_{L^2(S)}|^2dx,\qquad \beta_k= \int_{\R\setminus I(b)}|h(x-\lambda_k)|^2dx,
\]
so that (\ref{estimate for I}) may be presented as
\begin{equation}\label{estimate for I notation}
\Phi\leq C(\gamma)\Big(\sum_{j=1}^{N_I}\alpha_j\Big)^{\frac{1}{2}}\Big(\sum_{k=1}^{N_I}\beta_k\Big)^{\frac{1}{2}}.
\end{equation}

\underline{\textit{Step} V.} We wish to estimate $\alpha_j$ and $\beta_k$, $j,k=1,...,N_I$. To estimate $\alpha_j$ denote by $\phi_{\lambda_j}$ the Fourier transform of $P_Wg_{\lambda_j}$, and note that
\[
\langle g_{\lambda_j}, P_W e_x\rangle_{L^2(S)}=\langle P_Wg_{\lambda_j},  e_x\rangle_{L^2(S)} =\phi_{\lambda_j}(x).
\]
With this we have
\[
\alpha_j=\int_{\R\setminus I(b)}|\phi_{\lambda_j}(x)|^2dx\leq \|\phi_{\lambda_j}\|^2_{L^2(\R)}.
\]
Applying Parseval equality and the estimate (\ref{bound for g_l}) we find that
\begin{equation}\label{alpha}
\alpha_j\leq \|P_Wg_{\l_j}\|_{L^2(S)}^2\leq \|g_{\l_j}\|_{L^2(S)}^2\leq M^2.
\end{equation}

We turn to estimate $\beta_k$. For this, note that since $\l_k\in I$ we have $(-b,b)\subseteq I(b)-\l_k$. This implies that
\begin{equation}\label{beta}
\beta_k\leq \int_{|x|\geq b}|h(x)|^2dx\leq \epsilon^2,
\end{equation}
where the last estimate follows from (\ref{b-size}).

\underline{\textit{Step} VI.} To complete the proof we estimate $\Phi$ by inserting (\ref{alpha}) and (\ref{beta}) into (\ref{estimate for I notation}), for all $j,k=1,...,N_I$. We find that
\[
\Phi\leq C(\gamma)M N_I \epsilon.
\]
Recalling that $\Phi$ is the right most integral in (\ref{use-of-observ}), we obtain
\[
(1-d^2\gamma^2-C(\gamma)M\epsilon)N_I\leq |S|(|I|+2b)+1.
\]
Obtaining this estimate for all intervals $I$ of a fixed length $r$, taking the maximal value of $N_I$ over all such intervals, dividing by $r$ and allowing it to tend to infinity, we get
\[
(1-d^2\gamma^2-C(\gamma)M\epsilon)D^+(\Lambda)\leq |S|.
\]
Letting first $\epsilon$ tend to $0$, and then $\gamma$ tend to $1$, the result follows.
\end{proof}

\section{Uniform completeness}\label{section 3}
In this section we motivate the definition of uniformly complete systems. This definition appears to be more natural over the torus, and so we first focus our attention on that setting and study the concept there. In the second part of this section we give an example of a uniformly complete system over an interval which is not a frame there. 

\subsection{A discussion over the torus}
We denote  $\T=\R/\Z$ and identify $L^2(\T)$ with $L^2[0,1]$ as is the convention. Since the dual group of $\T$ is $\Z$, we restrict our attention to exponential functions with integer frequencies, and keep in mind that the family $\{e_n\}_{n\in\Z}$ forms an orthonormal basis in the space.

As pointed out in the introduction, the duality between Riesz sequences and frames is realized in several known results. The following lemma is a well known example over the torus. For a proof see e.g. \cite{OUbook}, p 10, Proposition 1.23.
\begin{lemma}\label{duality lemma}
Let $S\subseteq [0,1]$, $0<|S|<1$, and let $\L\subseteq\Z$. Denote $S^c=[0,1]\setminus S$ and $\L^c=\Z\setminus \L$. Then the following are equivalent.
\begin{itemize}
\item[i.] $E(\L)$ is a frame in $L^2(S)$.

\vspace{5pt}

\item[ii.] $E(\L^c)$ is a Riesz sequence in $L^2(S^c)$.
\end{itemize}
\end{lemma}

Our next goal is to study the notion of 'uniform completeness' presented in Definition \ref{def unif comp}. For this we first adapt this notion to the torus, that is, to the case where $S\subseteq [0,1]$ and the frequencies of the exponential functions are restricted to the integers. To avoid ambiguities, we term this adapted notion '($\Z$)-uniform completeness'.

\begin{definition}\label{torus unif comp}

Let $S\subseteq [0,1]$ be of positive measure and let $\L\subseteq\Z$. We say that $E(\L)$ is ($\Z$)-\textit{uniformly complete} in $L^2(S)$ if for every $k\in\Z$ and every $\epsilon>0$ there exists a linear combination
\[
p_k=\sum_{n\in\L} a_{n}(k)e_{n},
\]
 so that
$
\|e_{k}-p_k\|_{L^2(S)}<\epsilon
$
and
\begin{equation}\label{size of coef}
\sum_{n\in\L}|a_{n}(k)|^2\leq C,
\end{equation}
where $C$ is a positive constant which does not depend on $k$ or on $\epsilon$.
\end{definition}

The notions of '($\Z$)-uniform completeness' and 'uniform minimality' are related via the following lemma, which should be compared with Lemma  \ref{duality lemma}.

\begin{lemma}\label{duality lemma for unif}
Let $S\subseteq [0,1]$, $0<|S|<1$, and let $\L\subseteq\Z$. Denote $S^c=[0,1]\setminus S$ and $\L^c=\Z\setminus \L$. Then the following are equivalent.
\begin{itemize}
\item[i.] $E(\L)$ is ($\Z$)-uniformly complete in $L^2(S)$.

\vspace{5pt}

\item[ii.] $E(\L^c)$ is uniformly minimal in $L^2(S^c)$.
\end{itemize}
\end{lemma}

\begin{proof}
Let $S\subseteq [0,1]$ and $\L\subseteq\Z$ be as in the lemma. We first show that (\textit{ii}) implies (\textit{i}). Let $k\in\Z$. We wish to show that $e_k$ may be approximated as described in Definition \ref{torus unif comp}. If $k\in \L$ this is immediate, and so we may assume that $k\in \L^c$.

By the condition of uniform minimality in (\textit{ii}), there exists a function $g_k$, supported on $S^c$, so that
\begin{equation}\label{scalar product}
\langle g_k, e_{\ell} \rangle_{L^2[0,1]}=\langle g_k, e_{\ell} \rangle_{L^2(S^c)} =\delta_k(\ell)\qquad,\forall \ell\in\L^c.
\end{equation}
and
\begin{equation}\label{norm}
\|g_k\|_{L^2[0,1]}=\|g_k\|_{L^2(S^c)}\leq C,
\end{equation}
where $C$ is a constant which does not depend on $k$.

When considered as a function in $L^2[0,1]$, the function $g_k$ admits a Fourier decomposition. The relation (\ref{scalar product}) implies that this decomposition takes the form
\begin{equation}\label{gk series}
g_k=e_k+\sum_{n\in\L}a_n(k)e_n,
\end{equation}
where the equality holds in $L^2[0,1]$. As $g_k$ is supported on $S^c$, the decomposition (\ref{gk series}) implies that in $L^2(S)$ we have
\[
e_k=-\sum_{n\in\L}a_n(k)e_n.
\]
Moreover, by applying Parseval equality to  (\ref{gk series}), we find that
\[
\sum_{n\in\L}|a_n(k)|^2< \sum_{n\in\L}|a_n(k)|^2+1=\|g_k\|_{L^2[0,1]}^2\leq C,
\]
where the last inequality is due to (\ref{norm}).

This proves that $E(\L)$ is ($\Z$)-uniformly complete in $L^2(S)$. In fact, this seemingly proves a somewhat stronger property: $e_k$ is not only well approximated, but in fact it is well decomposed. See Remark \ref{unif comp decomp} below in this regard.

We next show that (\textit{i}) implies (\textit{ii}). Let $k\in\L^c$. Our goal is to construct a function $g_k$, supported on $S^c$, which satisfies (\ref{scalar product}) and (\ref{norm}).

Let $\epsilon>0$. The condition of ($\Z$)-uniform completeness implies the existence of a trigonometric polynomial $p_k^{(\epsilon)}$ \textit{with frequencies from $\L$} which satisfies
\[
\|e_k-p_k^{(\epsilon)}\|_{L^2(S)}\leq \epsilon,\quad\textrm{and}\quad \|p_k^{(\epsilon)}\|_{L^2[0,1]}\leq C,
\]
where the estimate on the right follows by applying Parseval's equality to (\ref{size of coef}), and so $C$ is a constant which does not depend on $k$ or on $\epsilon$.

Denote
$
g_k^{(\epsilon)}=e_k-p_k^{(\epsilon)}.
$
The properties of $p_k^{(\epsilon)}$ imply that 
\begin{equation}\label{app of ek}
\|g_k^{(\epsilon)}\|_{L^2(S)}\leq \epsilon,
\end{equation}
and,
\begin{equation}\label{norm via coef}
\|g_k^{(\epsilon)}\|_{L^2[0,1]}\leq C,
\end{equation}
where $C$ is a constant which does not depend on $k$ or on $\epsilon$.

Since the family $g_k^{(\epsilon)}$ is bounded in $L^2[0,1]$, there exists a subsequence $g_k^{(\epsilon_j)}$ which converges weakly to some function $g_k\in L^2[0,1]$ as $\epsilon_j\rightarrow 0$. Conditions (\ref{app of ek}) and (\ref{norm via coef}) ensure that $g_k$ is supported on $S^c$ and that $\|g_k\|_{L^2(S^c)}\leq C$.

Recall that for every $\epsilon>0$ the polynomials $p_k^{(\epsilon)}$ have spectrum supported on $\L$. It follows that for every $j$, we have
\[
\langle g_k^{(\epsilon_j)},e_{\ell}\rangle_{L^2[0,1]} =\langle e_k-p_k^{(\epsilon_j)}, e_{\ell}\rangle_{L^2[0,1]}=\delta_k(\ell)\qquad\forall\ell\in\L^c.
\]
As $g_k$ is a weak limit of $g_k^{(\epsilon_j)}$, and $g_k$ is supported on $S^c$, we may conclude that (\ref{scalar product}) holds for $g_k$. This completes the proof.

\end{proof}

\begin{remark}\label{unif comp decomp}
Let $S\subseteq [0,1]$ and $\L\subseteq \Z$ be as above. If $E(\L)$ is ($\Z$)-uniformly complete in $L^2(S)$, then for every $k\in\Z$ the function $e_k$ admits a decomposition
\[
e_k=\sum_{n\in\L} a_n(k)e_n
\]
which converges in $L^2(S)$ and which satisfies
\[
\sum_{n\in\L} |a_n(k)|^2<C,
\]
for some positive constant $C$ which does not depend on $k$. Indeed, this follows from an examination of the proof of Lemma  \ref{duality lemma for unif}.
\end{remark}

Fix $m>0$. One may apply a standard dilation argument to Definition \ref{torus unif comp} and to Lemma \ref{duality lemma for unif}, to obtain a notion of "$(\Z/m)$-uniform completeness" which holds for sets $S\subseteq [-m/2,m/2]$ of positive measure and for $\L\subseteq \Z/m$.
The extension to the notion of uniform completeness, as it appears in Definition \ref{def unif comp}, now seems natural.
\subsection{An example over the real line}\label{subsection 3.2}

Our goal in this subsection is to show that a uniformly complete exponential system need not be a frame. An analogues example for uniformly minimal systems is well known, and discussed in \cite{OUbook}, pp 10-11, (see also the references therein). We start by stating the latter example explicitly.

\begin{example}\label{unif min example} Let  $\Gamma=\{...-5,-3,-1, 2, 4, 6,...\}$. Then $E(\Gamma)$ is uniformly minimal in $L^2[0,1/2]$ but not a Riesz sequence there.
\end{example}

We now construct an analogues example for uniformly complete systems.

\begin{example}\label{unif comp example}  Let  $\Lambda=\{...-5,-3,-1,0, 2, 4, 6,...\}$. Then $E(\Lambda)$ is uniformly complete in $L^2[0,1/2]$ but not a frame there.
\end{example}

\begin{proof}
We make a preliminary observation: If $E(\L)$ is a frame in $L^2(S)$ then the same is true for $E(-\L)$, where $-\L:=\{-\l:\:\l\in\L\}$. A similar observation holds for all types of systems discussed in this note. Further, these systems are invariant also under translations of the set $S$.

By this observation, it is enough to show that $E(-\L)$ is not a frame in $L^2[1/2,1]$, and that it is uniformly complete in $L^2[1/4,3/4]$.

The first statement is well known (see e.g. \cite{OUbook}, pp 10-11), we add a short proof for completeness. Assume for a contradiction that $E(-\L)$ is a frame in $L^2[1/2,1]$. By Lemma \ref{duality lemma}, this implies that $E(\Gamma)$ is a Riesz sequence in $L^2[0,1/2]$, where
\[
\Gamma=\Z\setminus (-\L)=\{...-5,-3,-1, 2, 4, 6,...\}.
\]
This contradicts Example \ref{unif min example}.

We next show that $E(-\L)$ is uniformly complete in $L^2[1/4,3/4]$. Denote as above, $\Gamma=\Z\setminus(-\L)$. We start by observing that by Example \ref{unif min example} (and the preliminary observation at the beginning of this proof), $E(\Gamma)$ is uniformly minimal in $L^2[-1/4,1/4]$. Since all functions in the system $E(\Gamma)$ are $1$-periodic, we find that the system is uniformly minimal in $L^2([0,1/4]\cup[3/4,1])$ as well.

The last assertion implies that for every $k\in\Gamma$ there exists a function $g_{k}\in L^2[0,1]$, which satisfies $g_{k}=0$ almost everywhere on $[1/4,3/4]$, so that
\begin{equation}\label{scalar product ii}
\langle g_{k}, e_{\ell}\rangle_{L^2[0,1]} =\delta_{k}(\ell)\qquad \forall\ell\in\Gamma,
\end{equation}
and
\begin{equation}\label{norm ii}
\|g_{k}\|_{L^2[0,1]}\leq C,
\end{equation}
where $C$ is a positive constant which does not depend on $k$.

Let $\phi\in C^{\infty}(\R)$ be a function which satisfies $|\phi|\leq 1$ on [0,1], $\phi=1$ over $[1/4,3/4]$, and $\phi=0$ over $\R\setminus  [0,1]$. For such a function there exists a constant $M>0$ so that,
\begin{equation}\label{1-estimate}
\sum_{k\in\Z}|\widehat{\phi}(x+k)|\leq M\qquad \forall x\in\R.
\end{equation}

Fix $x\in\R$. We wish to show that $e_x$ can be 'well approximated' over $[1/4,3/4]$ by linear combinations of elements from $E(-\L)$. We will show that a stronger statement holds, namely, that in this case $e_x$ can be decomposed into a series with the required control over the size of the coefficients.

 Denote $\phi_x=e_{(-x)}\phi$, and note that $\widehat{\phi}_x(k)=\widehat{\phi}(x+k)$. By (\ref{norm ii}) and (\ref{1-estimate}) the series
\[
H_x:=\phi_x-\sum_{k\in\Gamma} \widehat{\phi}_x(k)g_{k}
\]
converges in $L^2[0,1]$, and moreover,
  \begin{equation}\label{bound for H}
\|H_x\|_{L^2[0,1]}\leq 1+ \sum_{k\in\Gamma}|\widehat{\phi_x}(k)|\|g_{k}\|_{L^2[0,1]}\leq 1+MC=C_1,
  \end{equation}
 where $C$ and $M$ are the constants from the estimates (\ref{norm ii}) and (\ref{1-estimate}) respectively, and so $C_1$ is a constant which does not depend on $x$.

 Consider the Fourier decomposition of $H_x$. For $\ell\in\Gamma$ the identity (\ref{scalar product ii}) implies that
  \[
 \widehat{H}_x(\ell)=\widehat{\phi}_x(\ell)- \sum_{k_\in\Gamma} \widehat{\phi}_x(k)\langle g_{k}, e_{\ell}\rangle_{L^2[0,1]} =0.
  \]
 The spectrum of $H_x$ is therefore supported on $-\L=\Z\setminus\Gamma$ and so its Fourier decomposition takes the form
  \begin{equation}\label{fourier decomp of H}
  {H}_x=\sum_{n\in(-\L)} a_n(x)e_n.
  \end{equation}

 Observe that on the interval $[1/4,3/4]$ we have $H_x=e_x$. Indeed, this follows from the fact that on this interval $\phi$ is equal to $1$, while the functions $g_k$ are equal to zero. Keeping in mind the decomposition (\ref{fourier decomp of H}) we conclude that on  $[1/4,3/4]$ we have 
 \[
  e_x=\sum_{n\in(-\L)} a_n(x)e_n.
 \]

 It remains to show the required control over the size of the coefficients. This follows by first applying Parseval's identity to the decomposition (\ref{fourier decomp of H})  and then applying the estimate (\ref{bound for H}), as follows,
 \[
 \sum_{n\in(-\L)}| a_n(x)|^2=\|{H}_x\|^2_{L^2[0,1]}\leq C_1^2.
 \]
 As $C_1$ is a constant not depending on $x$, the result follows.
\end{proof}

\section{Approximate uniform completeness}\label{section 4}

In this section we prove Theorem \ref{d-almos-unif-comp} and obtain Theorem \ref{thm unif complete} as a corollary. Our approuch follows an approach of Olevskii and Ulanovskii, where stability and duality arguments are applied to obtain density results (see e.g. \cite{OUbook}, p48).

\subsection{Auxiliary lemmas}

The following lemma is well known. It regards the so called \textit{Bessel property} of exponential systems. For a proof see e.g. \cite{OUbook}, p15.

\begin{lemma}\label{bessel lemma}
Let $S\subseteq [0,1]$ be of positive measure and let $\L\subseteq\R$ be a uniformly discrete sequence with separation constant $\delta>0$. Then for every $f\in L^2(S)$ we have
\[
\sum_{\l\in\L}|\langle f,e_{\l} \rangle_{L^2(S)}|^2\leq C(\delta)\|f\|^2_{L^2(S)},
\]
%Then for every $\{a(\l)\}\in\ell^2(\L)$ we have
%\[
%\|\sum_{\l\in\L}a(\l)e_{\l}\|^2_{L^2(S)}\leq C(\delta) \sum_{\l\in\L}|a(\l)|^2,
%\]
where $C(\delta)$ is a positive constant depending only on $\delta$.
\end{lemma}

We will make use of a standard stability argument. The version we require is formulated in the lemma below. We add a proof for completeness.

\begin{lemma}\label{stability}
Let $S\subseteq [0,1]$ be  of positive measure and let $\L\subseteq\R$ be a uniformly discrete sequence with separation constant $\delta>0$. Fix $0<\eta\leq \delta/3$. Assume that $\{\rho_{\l}\}_{\l\in\L}\subseteq \R$ satisfies $|\rho_{\l}-\l|\leq \eta$ for every $\l\in\L$. Then for every $\{a(\l)\}\in\ell^2(\L)$ we have
\[
\Big\|\sum_{\l\in\L}a_{\l}(e_{\rho_{\l}}-e_{\l})\Big\|^2_{L^2(S)}\leq C(\delta)\eta^2 \sum_{\l\in\L}|a_{\l}|^2.
\]
where $C(\delta)$ is a positive constant depending only on $\delta$.
\end{lemma}
\begin{proof}
We may assume that $\sum|a_{\l}|^2=1$. Denote
\[
H:=\sum_{\l\in\L}a_{\l}(e_{\rho_{\l}}-e_{\l}).
\]
The statement will follow if we show that for $f\in L^2(S)$ with $\|f\|_{L^2(S)}=1$ we have
\begin{equation}\label{needed for stability}
|\langle f,H\rangle_{ L^2(S)}|^2\leq C(\delta)\eta^2 .
\end{equation}
Note that the Fourier transform of any such $f$ is differentiable and moreover, since $S\subseteq [0,1]$, Parseval identity implies that $\|(\widehat{f})'\|_{L^2(\R)}^2\leq 4\pi^2$.
%\begin{equation}\label{norm of derivative}
%\|(\hat{f})'\|^2=4\pi^2\int_St^2|f(t)|^2\leq 4\pi^2,
%\end{equation}

We apply the Cauchy-Shwartz inequality, and then the intermediate value theorem, to get
\[
\begin{aligned}
|\langle f,H\rangle_{ L^2(S)}|^2&\leq \sum_{\l\in\L}|\widehat{f}(\rho_\l)-\widehat{f}(\l)|^2\leq \eta^2\sum_{\l\in\L}|(\widehat{f})'(\xi_\l)|^2,
\end{aligned}
\]
where $\xi_\l$ is a point in the interval between $\l$ and $\rho_{\l}$. In particular, this implies that $\{\xi_{\l}\}$ is uniformly discrete with separation constant $\delta/3$. The estimate (\ref{needed for stability}) now follows by applying Lemma \ref{bessel lemma} and recalling that $\|(\widehat{f})'\|^2\leq 4\pi^2$.
\end{proof}

\subsection{A proof for Theorem  \ref{d-almos-unif-comp}, necessity}

\begin{proof}
Let $S\subseteq\R$, $\L\subseteq\R$, and $0< d< 1$ be as in Theorem \ref{d-almos-unif-comp}. Note that if $S_1\subseteq S$ and $E(\Lambda)$ is uniformly complete over $S$ then it is also uniformly complete over $S_1$. We may therefore assume that $S$ is bounded. Further, by a standard dilation argument, we may assume that $S\subseteq [0,1]$.

Let $\delta>0$ be the seperation constant of $\L$ and fix $0<\eta\leq \min\{\delta/3,1\}$. We restrict our attention to the interval $I_{\eta}:=[0,1/\eta]$ and keep in mind that the normalized system $\sqrt{\eta}E(\eta\Z)$ forms an orthonormal basis in $L^2(I_{\eta})$.

For $\lambda\in\Lambda$, pick $\rho_{\lambda}\in \eta\Z$ so that $|\lambda-\rho_{\lambda}|\leq \eta$. Note that the numbers $\rho_{\lambda}$ are distinct and denote
\[
\L_{\eta}:=\{\rho_{\lambda}:\:\lambda\in\Lambda\}\qquad \Gamma_{\eta}=\eta\Z-\Lambda_{\eta}.
\]

 We claim that for an appropriate choice of $d_{\eta}$, the system $E(\Gamma_{\eta})$ is $d_{\eta}$ approximately uniformly minimal over $I_{\eta}\setminus S$. For this, let $k\in\Z$ be such that $\eta k\in\Gamma_{\eta}$. The $d$-approximate uniform completeness of $E(\L)$ over $S$ implies that there exist a sequence $\{a_{\lambda}(k)\}\in\ell^2(\L)$ and a function $\phi_k$ supported on $S$, so that
\begin{equation}\label{decomp of etak}
e_{\eta k}=\sum_{\lambda\in\Lambda}a_{\lambda}(k)e_{\lambda}+\phi_k\qquad \textrm{over } S,
\end{equation}
\begin{equation}\label{bound for phik}
\|\phi_k\|_{L^2(I_{\eta})}=\|\phi_k\|_{L^2(S)}\leq d\sqrt{|S|}
\end{equation}
and
\begin{equation}\label{bound for coef}
\sum_{\lambda\in\Lambda}|a_{\lambda}(k)|^2\leq C,
\end{equation}
where $C$ is a constant which does not depend on $k$  or  on $\eta$.

Denote
\[
h_k=1\!\!1_S\sum_{\lambda\in\Lambda}a_{\lambda}(k)\big(e_{\rho_{\lambda}}-e_{\lambda}\big).
\]
and note that due to Lemma \ref{stability} we have
\begin{equation}\label{norm of hk}
\|h_k\|^2_{L^2(I_{\eta})}= \|h_k\|^2_{L^2(S)}\leq C\eta^2.
\end{equation}
Next, define
\[
g_k=\eta1\!\!1_{I_{\eta}}\big[e_{\eta k}-\sum_{\lambda\in\Lambda}a_{\lambda}(k)e_{\rho_{\lambda}}+h_k-\phi_k\big].
\]

Our goal is to show that the family $\{g_k\}$ satisfies the requirements of Definition \ref{app unif min} with respect to the system $E(\Gamma_{\eta})$ and the set $I_{\eta}\setminus S$. We first note that due to (\ref{decomp of etak}) we have $g_k=0$ almost everywhere on $S$.

Next, we estimate the norm of $g_k$. Keeping in mind that $\sqrt{\eta}E(\eta\Z)$ is an orthonormal basis over $I_{\eta}$, we find that
\[
\begin{aligned}
\|g_k\|_{L^2(I_{\eta}\setminus S)}&=\|g_k\|_{L^2(I_{\eta})}\\
&\leq \sqrt{\eta}\big(1+\sum_{\lambda\in\Lambda}|a_{\lambda}(k)|^2\big)^{1/2}+\eta(\|\phi_k\|_{L^2(I_{\eta})}+\|h_k\|_{L^2(I_{\eta})}).
\end{aligned}
\]
Applying the estimates (\ref{bound for phik}), (\ref{bound for coef}) and (\ref {norm of hk}) we conclude that $\|g_k\|<C$, for some positive constant $C$ which does not depend on $k$.

Lastly, we evaluate the approximation (\ref{unif-min-up-to-d}) for the functions $g_k$. For $\ell\in\Z$ so that $\eta\ell\in\Gamma_{\eta}$ we have
\[
\begin{aligned}
\langle g_k, e_{\eta\ell}\rangle_{L^2(I_{\eta}\setminus S)}&=\langle g_k, e_{\eta\ell}\rangle_{L^2(I_{\eta})}\\
& =\delta_{\eta k}(\eta\ell)+\eta\langle h_k-\phi_k, e_{\eta\ell}\rangle_{L^2(I_{\eta})},
\end{aligned}
\]
and so
\[
\begin{aligned}
\sum_{\eta\ell\in\Gamma_{\eta}}|\langle g_k, e_{\eta\ell}\rangle_{L^2(I_{\eta}\setminus S)} -\delta_{\eta k}(\eta\ell)|^2&=\eta\sum_{\eta\ell\in\Gamma_{\eta}}|\langle h_k-\phi_k, \sqrt{\eta}e_{\ell\eta}\rangle_{L^2(I_{\eta})}|^2,\\
&\leq \eta\sum_{\ell\in\Z}|\langle h_k-\phi_k, \sqrt{\eta}e_{\eta\ell}\rangle_{L^2(I_{\eta})}|^2\\
&= \eta\|h_k-\phi_k\|_{L^2(I_{\eta})}^2\leq \eta\big(\|h_k\|_{L^2(I_{\eta})}+\|\phi_k\|_{L^2(I_{\eta})}\big)^2.
\end{aligned}
\]
This, combined with the estimates (\ref{bound for phik}) and (\ref {norm of hk}) implies that 
\[
\|\widehat{g}_k -\delta_{\eta k}\|_{\ell^2(\Gamma_{\eta})}\leq d_{\eta} 
\]
where
\[
d_{\eta}:=\sqrt{\eta}(d\sqrt{|S|}+\sqrt{C}\eta).
\]
We conclude that $E(\Gamma_{\eta})$ is $d_{\eta}$-approximately uniformly minimal in $L^2(I_{\eta}\setminus S)$.

We may now apply Theorem \ref{d-almos-unif-min} to get
\[
(1-d_{\eta}^2)D^+(\Gamma_{\eta})\leq \frac{1}{\eta}-|S|.
\]
Since $\Gamma_{\eta}\cup\Lambda_{\eta}=\eta\Z$ we have $D^+(\Gamma_{\eta})=1/\eta-D^-(\Lambda_{\eta})$. Moreover, as $\Lambda_{\eta}$ is a small perturbation of $\L$, we also have $D^-(\Lambda_{\eta})=D^-(\Lambda)$. Consequently, we get
\[
(1-d_{\eta}^2)\Big(\frac{1}{\eta}-D^-(\Lambda)\Big)\leq \frac{1}{\eta}-|S|.
\]
Rearranging, and inserting the expression for $d_{\eta}$, we obtain
\[
\begin{aligned}
|S|&\leq (1-d_{\eta}^2)D^-(\Lambda)+\frac{d_{\eta}^2}{\eta}\\
&\leq (1-d_{\eta}^2)D^-(\Lambda)+(d\sqrt{|S|}+\sqrt{C}\eta)^2
\end{aligned}
\]
Letting $\eta$ tend to zero, and noting that $d_{\eta}$ tends to zero as $\eta$ does, we get
\[
\begin{aligned}
(1-d^2)|S|\leq D^-(\Lambda).
\end{aligned}
\]
This completes the proof.
\end{proof}

\begin{remark}\label{d=0}
Theorem \ref{thm unif complete} follows from Theorem \ref{d-almos-unif-comp} by applying the latter theorem for all $d>0$, and then letting $d$ tend to zero.
\end{remark}

\subsection{A proof for Theorem  \ref{d-almos-unif-comp}, sharpness}
We now show that Theorem  \ref{d-almos-unif-comp} is sharp. Given $0< d< 1$ and $\epsilon>0$ we construct a set $S$ so that $E(\Z)$ is $d$-approximately uniformly complete in $L^2(S)$ and $1-\epsilon\leq {|S|}({1-d^2})$. This observation was obtained jointly with Rohit Pai.

For $n\in\N$  let
\[
\alpha_n=\frac{d^2}{n(1-d^2)+1},
\]
and fix $n$ large enough so that $\alpha_n\leq \epsilon$.

Denote
\[
S:=[0,1]\cup \big(\cup_{j=1}^{n}[j,j+\alpha_n]\big),
\]
and note that
\begin{equation}\label{|S| and d2|S|}
|S|=1+n\alpha_n,\qquad\textrm{and}\qquad d^2|S|=(n+1)\alpha_n.
\end{equation}

We claim that $E(\Z)$ is $d$-approximately uniformly complete in $L^2(S)$. To see this, given $x\in\R$ consider the $1$-periodic function $\phi_x$ which satisfies $\phi_x=0$ on $[0,\alpha_n]$ and $\phi_x=e_x$ on $[\alpha_n,1]$. The Fourier decomposition of $\phi_x$ over $[0,1]$,
\begin{equation}\label{fourier of phi}
\phi_x=\sum_{n=-\infty}^{\infty} a_n(x)e_n,
\end{equation}
satisfies  
\[
\sum_{n=-\infty}^{\infty} |a_n(x)|^2=\|\phi_x\|^2_{L^2[0,1]}\leq 1.
\]

As $\phi_x$ is $1$-periodic, the series on the right hand side of (\ref{fourier of phi}) converges to $\phi_x$ in $L^2(S)$. This implies that the series is equal to $e_x$ on $[\alpha_n,1]$ and equal to zero on $S\setminus [\alpha_n,1]$. By  (\ref{|S| and d2|S|}) we have
\[
\begin{aligned}
\Big\|e_x-\sum_{n=-\infty}^{\infty} a_n(x)e_n\Big\|^2_{L^2(S)}=(n+1)\alpha_n=d^2|S|.
\end{aligned}
\]
It follows that $E(\Z)$ is $d$-approximately uniformly complete in $L^2(S)$. Finally, we observe that by (\ref{|S| and d2|S|}) we have
\[
\begin{aligned}
{|S|}({1-d^2})&=|S|-d^2|S|=1-\alpha_n\geq 1-\epsilon.\\
\end{aligned}
\]
This completes the proof.

\section{Acknowledgment} The author thanks Rohit Pai for the fruitful conversation which resulted in the proof that Theorem  \ref{d-almos-unif-comp} is sharp.

\bibliographystyle{amsplain}
%\bibliography{../../../@bibliotek/bibliotek}

\def\cprime{$'$} \def\cprime{$'$} \def\cprime{$'$} \def\cprime{$'$}
\providecommand{\bysame}{\leavevmode\hbox to3em{\hrulefill}\thinspace}
\providecommand{\MR}{\relax\ifhmode\unskip\space\fi MR }
\providecommand{\MRhref}[2]{%
  \href{http://www.ams.org/mathscinet-getitem?mr=#1}{#2}
}
\providecommand{\href}[2]{#2}

\end{document}